\documentclass[11pt]{amsart}
\usepackage{amsmath, amssymb, amsfonts, amsthm, verbatim, color, mathtools}
\usepackage{geometry}
\usepackage[bookmarks]{hyperref}
\usepackage{mathtools}
\mathtoolsset{showonlyrefs}

\newtheorem{theorem}{Theorem}
\newtheorem{proposition}[theorem]{Proposition}
\newtheorem{corollary}[theorem]{Corollary}
\newtheorem{lemma}[theorem]{Lemma}

\newtheorem{remark}[theorem]{Remark}

\newcommand{\norm}[1]{\left\Vert#1\right\Vert}
\newcommand{\abs}[1]{\left\vert#1\right\vert}

\newcommand{\R}{{\mathbb R}}
\newcommand{\bN}{{\mathbb N}}
\newcommand{\bC}{{\mathbb C}}

\newcommand{\Z}{{\mathbb Z}}

\newcommand{\bP}{{\mathbb P}}

\newcommand{\cA}{{\mathcal{A}}}

\newcommand{\eps}{\varepsilon}

\def\be{\begin{equation}}
\def\ee{\end{equation}}

\def\o{\omega}

\numberwithin{theorem}{section}
\numberwithin{equation}{section}
\numberwithin{figure}{section}

\author{}
\date{}

\begin{document}

\title[Almost sure bounds of iterates for DNLW]{Almost sure boundedness of iterates for derivative nonlinear wave equations}
\author[S. Chanillo]{Sagun Chanillo}
\address{Department of Mathematics, Rutgers University
Piscataway, NJ 08854-8019, USA}
\email{chanillo@math.rutgers.edu}
\thanks{S.C. is funded in part by NSF DMS-1201474.}

\author[M. Czubak]{Magdalena Czubak}
\address{Department of Mathematics, University of Colorado at Boulder,
 Campus Box 395
Boulder, CO 80309-0395
USA}
\email{czubak@colorado.edu}
\thanks{M.C. is funded in part by the Simons Foundation \#246255.}

\author[D. Mendelson]{Dana Mendelson}
\address{ 
Department of Mathematics, University of Chicago, 5734 S. University Avenue, Chicago, IL  60637 }
\email{dana@math.uchicago.edu}

\thanks{D.M. was funded in part by NSF DMS-1128155 during the completion of this work.}

\author[A. Nahmod]{Andrea Nahmod}
\address{Department of Mathematics, 
University of Massachusetts Amherst, 
  710 North Pleasant St. 
Lederle Graduate Research Tower
Amherst, MA 01003-9305, USA}
\email{nahmod@math.umass.edu}
\thanks{A.N. is funded in part by NSF DMS-1201443 and DMS-1463714.}

\author[G. Staffilani]{Gigliola Staffilani}
\address{Department of Mathematics, 
Massachusetts Institue of Technology,  
77 Massachusetts Ave, Room 2-251, 
Cambridge, MA 02139-4307, 
USA}
\email{gigliola@math.mit.edu}
\thanks{ G.S. is funded in part by NSF DMS-1362509, DMS-1462401, John Simon Guggenheim Foundation, and the Simons Foundation.}

\begin{abstract}
We study nonlinear wave equations on  $\R^{2+1}$ with quadratic derivative nonlinearities, which include in particular nonlinearities exhibiting a null form structure, with random initial data in $H_x^1\times L^2_x$. In contrast to the counterexamples of Zhou \cite{Zhou} and Foschi-Klainerman \cite{FK}, we obtain a uniform time interval $I$ on which the Picard iterates of all orders are almost surely bounded in $C_t(I ; \dot H_x^1)$.

\end{abstract}
\maketitle

\section{Introduction}

We consider systems of derivative nonlinear wave equations given by
 \begin{align}
\label{e1_main}
 \Box u^I=\sum_{J, K} \partial u^J \partial u^K,
 \end{align}
where we use $\partial$ to denote a first order derivative, i.e. $\partial \in \{\partial_t, \partial_{x_1}, \partial_{x_2} \}$. In particular, this class of equations includes the so-called null-form derivative nonlinear wave equations
 \begin{align}\label{e1}
 \Box u^I=\sum_{J, K} Q(u^J, u^K),
 \end{align}
where $\Box=-\partial_t^2 +\Delta$,  \, $(u^I): \R^2\times \R \rightarrow \R^m$, $m\in {2, 3, \dots}$ and $Q$ is a bilinear form given by
\begin{align*}\label{null_forms0}
&Q_0(f, g) =\partial_t f\partial_t g-\nabla f\cdot \nabla g, \\
&Q_{0j}(f, g) =\partial_t f\partial_j g-\partial_j f \partial_t g, \\
&Q_{ij}(f, g) =\partial_i f\partial_j g-\partial_j f \partial_i g, 
\end{align*}
where $\partial_j$ stands for the spatial derivatives and $\nabla$ for the spatial gradient.  

We now give some history of the wave equations with null forms as well as of the more general derivative nonlinear wave equations.   The null forms were introduced by Klainerman \cite{Klainerman83, Klainerman84} (see also Christodoulou \cite{Christodoulou86}) and nonlinearities exhibiting quadratic derivatives with such null form structure  appear in many physical models and geometric wave equations, such as wave maps, the Maxwell-Klein-Gordon system, the Yang-Mills equations, the space-time Monopole equation, and Ward wave maps.

Wave maps are one of the simplest geometric wave equations and can be viewed as the Minkowski's analogue of a harmonic map. We refer the reader, for example to \cite{Tao_WMII, SS_WM, NSU, Struwe,  Krieger_H2,  RodnianskiSterbenz, SterbenzTataru_WM2, KriegerSchlag, CKLS2}  for some of the pioneering works.  The equations are particularly interesting in the two dimensional case which is energy critical. 

 The Maxwell-Klein-Gordon system can be viewed as an abelian analogue of Yang-Mills. Fundamental contributions to the mathematical analysis of the Yang-Mills gauge theory were made by Uhlenbeck \cite{Uhlenbeck_YM, Uhlenbeck_Lp},  see also the book of Freed and Uhlenbeck \cite{FreedUhl} for further discussion. We also refer the reader to  \cite{KMYM, KL_MKG, KrSchTa,  KST, OT_MKG, KT_YM, MoncriefMKG, CP, Pecher_MKGtemp} and references therein for more on the Maxwell-Klein-Gordon and Yang-Mills equations.
 
The Ward wave map and the space-time Monopole equation were introduced by Ward in \cite{Ward88,Ward89}. The first one is an integrable system in $2+1$ dimensions while the second one is a space-time analogue of Bogomolny equations. The Monopole equation can be derived from the anti-self dual Yang-Mills by a dimensional reduction, and the Ward wave map from the Monopole equation by choosing a particular gauge.  For results on the Cauchy problem for Ward wave maps and soliton construction see \cite{DaiTerng, CzubakT, Wu_WWM, GN}.  A broad survey on the space-time Monopole equation is given by Dai, Terng and Uhlenbeck \cite{DCU}, where, in particular, they show global existence and uniqueness up to a gauge transformation for small initial data in $W^{2,1}$ via a scattering transform.  For results on the Cauchy problem for the Monopole equation, see \cite{CzubakME, BournaveasCandy_ME, Tesfahun}.

\medskip
 
Our goal in this note is to study \eqref{e1_main} from a probabilistic point of view. Our motivation for this is two-fold. First, as we will see when we detail the deterministic local well-posedness theory below, in two dimensions there exists a gap in Sobolev space between the scaling invariant space and the optimal local well-posedness for general quadratic derivative nonlinearities \eqref{e1_main}, including null-form nonlinear wave equations with $Q_{ij}$ and $Q_{0j}$ nonlinearities.  In this work we make progress in bridging this gap.  A second reason we have for studying \eqref{e1_main}  is that it serves as a model problem for  \eqref{e1} and the geometric flows described above, and hence this note fits into a broader program of adapting probabilistic techniques to geometric equations. In many of these cases, however, one would need to define a new randomization procedure which respects the geometry of the target manifold, is compatible with gauge transformations and pull-backs, and which still yields (almost surely) better properties for the solutions. This is largely uncharted territory and represents an exciting future direction of research which will require many new ideas.

\medskip
We review the well-posedness results for \eqref{e1_main} and \eqref{e1}. Schematically, we compare these to the following derivative nonlinear wave equation:
\be\label{e2}
\Box u=(\partial u)^2,
\ee
where $u:\R^{d+1}\rightarrow \R$, and $\partial$ is any of the derivatives $\partial_\alpha$,\, $\alpha=0,\dots, d$.  Solutions to \eqref{e2} are invariant under the scaling transformation
\[
u \mapsto u_{\lambda}(t, x)=u(\lambda t, \lambda x)
\]
and the scaling critical regularity, $s_c=\frac d2$, is by definition the regularity such that the corresponding homogeneous Sobolev space $\dot H^{s_c}(\R^d)$ is invariant under the above scaling.  In particular, in dimension $d = 2$, one has $s_c = 1$, and hence \eqref{e2} is energy critical. Energy methods yield local well-posedness for $s>\frac d2+1$, and using Strichartz  estimates (see \cite{PonceSideris}) one may improve this to  $s> \max (\frac d2, \frac{d+5}{4})$, which is sharp for \eqref{e2} with data in the Sobolev class in light of the counterexamples of Lindblad \cite{Lindblad96}. Nonetheless, this leaves a gap in dimensions $d=2,3,4$ between the optimal regularity in Sobolev spaces and the scaling prediction. In dimension $d=3$, this gap was closed for data in the Fourier-Lebesgue class  \cite{Grunrock}. See also \cite{DFS_10, Tanguay, GrigoryanTanguay} for improved estimates in 2D.

By replacing the general product $(\partial u)^2$ by one of the null forms, Klainerman and Machedon demonstrated in \cite{KlainermanMachedon93} that one may lower the regularity required for local well-posedness from $s>2$ to $s=2$ for either $Q_0$, $Q_{0j}$ or $Q_{ij}$ in dimension $d=3$. Subsequently,  it was shown in \cite{KlainermanMachedon95, KS2, KS} that  in the particular case of \eqref{e1} with a $Q_0$ nonlinearity local well-posedness holds in $H^s(\R^d) \times H^{s-1}(\R^d)$ with $s>\frac{d}{2}$ and $d \geq 2$, which is almost optimal.  Similarly, for the $Q_{0j}$ and $Q_{ij}$ null form nonlinearities,  almost optimal local well-posedness can also be achieved in Sobolev spaces with index $s>\frac{d}{2}$ but in dimensions $d\geq 3$ \cite{KM3} .
 
In the energy critical setting $d=2$, the situation is more delicate for the $Q_{0j}$ and $Q_{ij}$ null form nonlinearities. The best known result in Sobolev spaces was achieved by Zhou \cite{Zhou}, who proved local well-posedness for $s>\frac 54$, which we note is still $\frac 14$ away from the scaling critical regularity. Examination of the first iterate shows that in  Sobolev spaces, this is optimal using iteration methods, see \cite{Zhou, KS, FK} for further discussion. This gap in dimension $d=2$ was closed  for data in the Fourier-Lebesgue class  in \cite{GN}.

\medskip
As mentioned above, our aim is to close this gap between the local well-posedness theory and the scaling critical regularity in Sobolev spaces with the aid of probabilistic methods. The use of probabilistic techniques to study the well-posedness of nonlinear dispersive equations was initiated by Bourgain motivated by the question of the invariance of associated Gibbs measures{\footnote{ after the works by Lebowitz, Rose and Speer \cite{LRS}, by Glimm and Jaffe \cite{GlimmJaffe} and by Zhidkov \cite{Zh1, Zh2}}. In \cite{B94, B96} Bourgain studied the periodic nonlinear Schr\"odinger equation in one and two space dimensions, in \cite{BGP} the Gross-Pitaevskii equation and in \cite{B94Z} the Zakharov system. In particular for the periodic cubic nonlinear Schr\"odinger equation in two dimensions, Bourgain constructed, for the first time, sets of supercritical initial data of close to full measure which give rise to local in time solutions, and then proved almost sure global well-posedness via the invariance of the Gibbs measure.
   
Following the seminal contributions of Bourgain, extensive progress has been made in recent years in the study of random data well-posedness of nonlinear dispersive, wave and fluid equations, both in compact and non-compact settings.  In the context of nonlinear wave equations, and after work by Zhidkov \cite{Zh2},  Burq and Tzvetkov considered in  \cite{BT1, BT2} the cubic nonlinear wave equation on a three-dimensional compact manifold. As in the two dimensional result of Bourgain, Burq and Tzvetkov also constructed sets of supercritical initial data of close to full measure which give rise to local solutions, and subsequently proved the almost sure existence of global solutions for the radial problem via the invariance of the associated Gibbs measure. For other probabilistic results on nonlinear wave equations see for instance \cite{BT3, BT4, Suzzoni1, Suzzoni2, LM1, LM2, Pocovnicu, OP, DLuM1} and references therein. 

As the randomization of the initial data does not regularize in Sobolev spaces, the free evolution of the random initial data is almost surely no more regular than the function which was randomized. Consequently, the typical scheme employed for the random data local well-posedness theory in previous works on nonlinear wave and dispersive equations with power-type nonlinearities has involved re-centering the flow and proving a fixed point argument for the nonlinear component of the solution. Carrying out this method requires one to gain regularity for the nonlinear component of the solution, and it is not clear if this can be done in the current setting of a quadratic derivative nonlinearity. We also note that it is sometimes possible to further re-center the initial data in cases where regularity is only gained for certain components of the nonlinearity, see for instance the recent work of  \cite{BOP}.

 \subsection*{Overview of the main results} We state the main theorem precisely in  Section \ref{sec:set-up}.  Here we just give the main idea.  We study a scalar version of equation  \eqref{e1_main} given by
\be\label{e3}
\Box u=(\partial u)^2.
\ee
We consider a pair of functions $(\phi_0, \phi_1) \in H^1(\R^2) \times L^2(\R^2)$, and we randomize them according to a unit-scale projection in frequency space. See Section \ref{sec:set-up} for the precise definition of the randomization.  We then consider the Picard iterates for \eqref{e3}, and show that the $n$-th iterate, $u^{(n)}(x,t)$, is almost surely bounded, that is, there exists $T$ sufficiently small such that with probability one, we have
\[
u^{(n)}(x,t) \in C_t\bigl([0,T]; \dot H^1(\R^d)\bigr),
\]
for every $n \geq 0$;  see Theorem \ref{main} for the precise statement. The analogous result can be also obtained for vector-valued equations.  See Remark \ref{vectorremark}.  

This result should be held in contrast to the deterministic results of Zhou which indicate that for initial data $\varphi \in  H^s(\R^2)$, there is no estimate for the first iterate for $s \leq \frac{5}{4}$.   More precisely, in \cite[Proposition 5]{Zhou} Zhou proves that given functions $f$ and $g$ which lie in $H^{s+1}(\mathbb{R}^2)$ for $0 \leq s \leq \frac{1}{4}$, if we consider
\[
F = \cos (t|\nabla|)f, \quad  G = \cos(t |\nabla|)g
\]
then the solution $\psi$ to
\begin{align}
\Box \psi &= Q_{12}\bigl(F, G \bigr),\\
(\psi, \psi_t)\big|_{t=0} &= (0,0),
\end{align}
fails to be in $H^{s+1}(\mathbb{R}^2)$, and in particular, the estimate
\[
\|\psi(t, \cdot)\|_{H^{s+1}}+ \|\partial_t \psi(t, \cdot)\|_{H^s}  \leq C(t) \|f\|_{H^{s+1}} \|g\|_{H^{s+1}}
\]
fails. 

\subsection*{Notation and conventions}
We denote by $C > 0$ an absolute constant that depends only on fixed parameters and whose value may change from line to line. We write $X \lesssim Y$ if $X \leq C Y$ for some $C > 0$, and analogously for $X \gtrsim Y$. Finally, we will use $\mathcal{F} f$ or  $\widehat{f}$ to denote the Fourier transform of a function $f$.

\subsection*{Organization of paper}
In Section \ref{sec:set-up}, we introduce our framework and give a precise statement of the main result, Theorem \ref{main}. In Section \ref{sec:prelim}, we collect some  probabilistic facts which we will need in our analysis. In Section \ref{sec:scheme}, we set up the iteration scheme. In Section \ref{sec:prob_bds}, we establish the main probabilistic bounds for the iterates, and finally, in Section \ref{sec:proof_main}, we prove the main theorem, Theorem \ref{main}.

\section{Set-up and statement of main results}\label{sec:set-up}
We introduce the randomization procedure by which we construct the initial data. We will present the randomization for real-valued initial data which readily generalizes to targets $\R^m$ by randomizing each coordinate according to the procedure described below.

To define the randomized initial data on Euclidean space, we let $\psi \in C_c^{\infty}(\R^2)$ be a non-negative function with $\text{supp} (\psi) \subset B(0,1)$ and such that 
\[
 \sum_{k \in \Z^2} \psi(\xi - k) = 1 \text{ for all } \xi \in \R^2.
\]
For every $k \in \Z^2$, we define the function $P_k f: \R^2 \rightarrow \bC$ for any $f \in H^{s}(\R^2)$  by
\begin{equation*}
 (P_k f)(x) = \mathcal{F}^{-1} \left( \psi(\xi - k) \hat{f}(\xi) \right)(x) \text{ for } x \in \R^2.
\end{equation*} 
 We set $\phi = (\phi_0, \phi_1) \in H^1(\R^2) \times L^2(\R^2)$ and we let $\{ (\varepsilon_k, \nu_k) \}_{k \in \Z^2}$  be a sequence of identically distributed independent Rademacher random variables on a probability space $(\Omega, {\mathcal A}, \bP)$. We recall that by definition these are centered random variables which take values $\pm1$ with equal probability. Define
\be\label{initial_eucl}
\phi^\omega =\bigl( \phi_0^\omega(x), \phi_1^\omega(x)\bigr) :=  \biggl(\sum_{k\in \Z^2} \varepsilon_k(\omega) P_k\phi_0(x) ,  \sum_{k\in \Z^2} \nu_k(\omega) P_k \phi_1(x) \biggr),
\ee
where this quantity is understood as a Cauchy limit in $L^2(\Omega; H^{1}(\R^2) \times L^2(\R^2))$. Similar randomizations have previously been used in \cite{ZF}, \cite{LM1}, \cite{BOP1}. 
Crucially, such a randomization does not regularize at the level of Sobolev spaces, see \cite{BT1}.

\subsection{Statement of main results}
We denote the free wave evolution of the initial data $\phi^\omega$ by
\begin{align} \label{equ:free_evolution}
W(t)  \phi^\o = \cos(t|\nabla|) \phi_0^{\o} + \frac{\sin(t|\nabla|)}{|\nabla|} \phi_1^{\o}.
\end{align}

Since the randomization does not regularize in Sobolev spaces, the free evolution of the random initial data \eqref{equ:free_evolution} remains at regularity $H^{1}(\R^2) \times L^2(\R^2)$. Typically, when studying such problem, one writes solutions as
\[
(u(t), \partial_t u(t) ) = \bigl(W(t)  \phi^\o, \partial_t W(t)  \phi^\o\bigr) +(w(t), \partial_t w(t)),
\]
and proves that the nonlinear component $(w(t), \partial_t w(t))$ is almost surely smoother than the random initial data (this dates back to \cite{B96}). In studying \eqref{e3},  due to the derivative structure of the nonlinearity, it is not clear that the solution lies in a smoother space. We instead turn to the Picard iterates. We define 
\begin{align} \label{equ:zeroth_iterate}
u^{(0)}(x,t) = W(t) (\phi_0^\omega, \phi_1^\omega)
\end{align}
and, setting the notation
\begin{align}
A_0(\cdot, \cdot) = \int_0^t \frac{\partial \sin((t - t') |\nabla|)}{|\nabla|} (\cdot) (\cdot) dt',
\end{align}
we define the $n$-th iterate inductively as
\begin{align} \label{nth}
\partial u^{(n)} = \partial u^{(0)} + A_0(\partial u^{(n-1)}, \partial u^{(n-1)}), \quad \partial u^{(0)} = \sum_{k} \varepsilon_k(\omega) P_k \partial  W(t) \phi =: \sum_{k} \varepsilon_k(\omega) F_k.
\end{align}
Our goal is to prove that almost surely, the iterates for \eqref{e3} are bounded.  We are now prepared to state our main theorem.
\begin{theorem}\label{main}
Let $(\Omega, {\mathcal A}, \bP)$ be a probability space.  Let $(\phi_0, \phi_1) \in H^1(\R^2) \times L^2(\R^2)$ and let $\phi^\omega$ be as given in \eqref{initial_eucl}. Let $u^{(0)}$ be the free evolution of $\phi^\omega$ defined in \eqref{equ:zeroth_iterate} and let the $n$-th Picard iterate be as defined in \eqref{nth}. Then for any $T > 0$ sufficiently small, there exists $\Sigma_T \subseteq \Omega$, with $\mathbb P(\Sigma_T)=1$, such that for every $\omega \in \Sigma_T$ and $n\geq 0$, 
\[
\bigl(u^{(n)}, \partial_t u^{(n)}\bigr) \in C_t \bigl([0,T]; \dot H^1(\R^2) \times L^2(\R^2) \bigr).
\]
\end{theorem}

\begin{remark}\label{vectorremark}
Although we prove this for scalar valued equations, the results generalize to vector-valued equations by considering the terms component-wise.
For general initial data with an $m$-dimensional target, we consider a probability space $(\Omega, \cA, \bP)$ and we let $\{ (\eps_k, \nu_k) \}_{k \in \Z^2}$ be a sequence of \textit{$m$-dimensional} Rademacher random variables. We then randomize the initial data with multi-dimensional target by randomizing each coordinate separately.  In particular, Theorem \ref{main} yields the same results for systems of derivative nonlinear wave equations, as in \eqref{e1_main}.
\end{remark}

 Our proof relies on writing the $n$-th iterate in terms of not only the previous iterate, but on \emph{un-packing} it all the way to the free solution.  Because the nonlinearity is quadratic, this results in $2^n$ many unit-sized projections of the free solutions interacting together at the $n$-th iterate level.  By expanding the iterates completely, the randomization will then enable us to use the Bernstein inequality of Lemma \ref{bernstein} below to take advantage of the support properties of each term. We will also show that the quadratic structure  naturally gives rise to a connection with binary trees. 
 
\begin{remark}
An almost identical proof will yield the result in Theorem \ref{main} for dimensions three and four, the only modification being the factor appearing in Proposition \ref{prop:g_bds2} arising from Bernstein's inequality, see Remark \ref{bern_dim}.
\end{remark}

\section*{Acknowledgements}
 The authors thank the MSRI, the IHES and the Radcliffe Institute for Advanced Study at Harvard University for the kind hospitality that allowed us to develop this project.  D.M. gratefully acknowledges support from the Institute for Advanced Study
at Princeton, and would also like to thank Carlos Kenig for helpful conversations.

\section{Preliminaries} \label{sec:prelim}
Here we record some facts which will be of use to us in our analysis. The first is a unit-scale Bernstein estimate for the projection operators.

\begin{lemma}[Bernstein inequality \protect{\cite{Schlag_notes}}]\label{bernstein}
Let $1 \leq q \leq p \leq \infty$. There exists a constant $C_0$, depending only on the dimension $n$, and $p, q$ such that if $f$ is a function with a support of $\hat f$ contained in a measurable set $E$, then
\be
\|  f \|_{L^p_x} \leq C_0 \abs{E}^{\frac 1q-\frac 1p} \|  f\|_{L^q_x}.
\ee
\end{lemma}
\begin{remark}
We apply this in two situations.  When we have a single function $P_k f$, and when we have a product of $j$ functions $P_{k_{i}}f$ for $i=1, \ldots, j$ and $j\geq 2$.  In both of these cases, we will take $p=4$ and $q=2$.  This leads to
\be
\| P_k f \|_{L^4_x} \leq C_0 \pi^\frac 14\| P_k f\|_{L^2_x},
\ee
and 
\be
\| P_{k_1} f\cdots P_{k_j} f \|_{L^4_x} \leq C_0 \pi^\frac 14 j^\frac 12\| P_k f\|_{L^2_x},
\ee
since the supports are then contained in the ball of radius $1$ and $j$, respectively.
\end{remark}

In the sequel, it will be useful to introduce the notation
\be\label{m01}
M_{01} := \frac{\partial \sin((t_0 - t_1) |\nabla|)}{|\nabla|},
\ee
where once again we mention that $\partial$ can be either a spatial derivative or a time derivative. We will repeatedly make use of the following fact about the multiplier $M_{01}$.

\begin{lemma} \label{lem:mult_bds}
The multiplier $M_{01}$ is bounded on $L_x^2$, with
\begin{align}
\|M_{01} f \|_{L^2_x} \leq \|f\|_{L^2_x}.
\end{align}
\end{lemma}

\subsection{Large deviation estimates}\label{sec:ldev}
We begin by recalling a large deviation estimate, which goes back to the classical work of Kolmogorov, Paley and Zygmund.

\begin{lemma}[\protect{\cite[Lemma 3.1]{BT1}}] \label{lem:large_deviation_estimate}
 Let   $\{\varepsilon_k\}_{k=1}^{\infty}$ be a sequence of independent identically distributed (iid) Rademacher random variables on a probability space $(\Omega, {\mathcal A}, \bP)$. Then there exists $C > 0$ such that for every $p \geq 2$ and every $\{c_k\}_{k=1}^{\infty} \in \ell^2(\bN; \bC)$, we have
 \begin{equation}\label{bdeterm}
  \Bigl\| \sum_{k=1}^{\infty} c_k \varepsilon_k(\omega) \Bigr\|_{L^p_\omega(\Omega)} \leq C \sqrt{p} \Bigl( \sum_{k=1}^{\infty} |c_k|^2 \Bigr)^{\frac{1}{2}}.
 \end{equation}
As a consequence of Chebychev's inequality, there exists $\alpha > 0$ such that for every $\lambda > 0$ and every sequence $\{c_k\}_{k=1}^{\infty} \in \ell^2(\bN;\bC)$ of complex numbers, 
 \begin{equation*}
  \bP \Bigl( \bigl\{ \omega : \bigl| \sum_{k=1}^{\infty} c_k \varepsilon_k (\omega) \bigr| > \lambda \bigr\} \Bigr) \leq 2 \exp \biggl(- \alpha \frac{\lambda^2}{\sum_k |c_k|^2} \biggr). 
 \end{equation*}
\end{lemma}
We will also use the following lemma, which can be viewed as a generalization of estimate \eqref{bdeterm} which allows for nondeterministic coefficients. A version of this lemma appeared in \cite{ErdosYau}, however we include here a proof of this fact for completeness.
\begin{lemma}\label{blemma}
 Let $\{\varepsilon_k\}_{k=1}^{\infty}$ be a sequence of iid Rademacher random variables on a probability space $(\Omega, {\mathcal A}, \bP)$, and let $b_k$ be a sequence of random variables which are independent of the $\{\varepsilon_k\}_{k=1}^{\infty}$. Then there exists $C > 0$ such that for all $p \geq 2$, 
\begin{align}\label{best}
\bigl\|  \sum_{k=1}^\infty   \varepsilon_{k} b_{k}  \bigr\|_{L^p_\omega(\Omega)} &\leq C \sqrt{p} \biggl\| \biggl( \sum_{k=1}^\infty     |b_{k}|^2 \biggr)^{\frac{1}{2}}  \biggr\|_{L^p_\omega(\Omega)}.
\end{align}
\end{lemma}
\begin{proof}
 First we observe that by the Monotone Convergence Theorem and the Cauchy criterion for convergence of infinite series, it is enough to prove the estimate for a finite sum.  
 
 Next we will prove the desired inequality for even powers of $p$, i.e., $p=2j\geq 2$. Using the independence and that $\mathbb E( \varepsilon_k) =0$, one can easily see that 
\be\label{best1}
\bigl\|  \sum^N_{k=1}   \varepsilon_{k} b_{k}  \bigr\|_{L^p_\omega(\Omega)}^p=\sum_{2k_1+\cdots + 2k_N=2j} \int \frac{(2j)!}{(2k_1)!\cdots (2k_N)!} \abs{b_{1}}^{2k_1}\cdots \abs{b_{N}}^{2k_N}.
\ee
Expanding the right hand side  of \eqref{best}, we similarly get for $p=2j$ that 
\be\label{best2}
\biggl\| \biggl( \sum_{k=1}^N     |b_{k}|^2 \biggr)^{\frac{1}{2}}  \biggr\|^p_{L^p_\omega(\Omega)}=\sum_{k_1+\cdots + k_N=j} \int \frac{(j)!}{k_1!\cdots k_N!} \abs{b_{1}}^{2k_1}\cdots \abs{b_{N}}^{2k_N}.
\ee
Thus, comparing the right hand side of \eqref{best1} with \eqref{best2} gives us an estimate with constant 
\begin{equation}\label{normconstant}
\max_{j_1, \dots, j_N}  \frac{(2j)!}{j!}  \frac{k_1!\cdots k_N!}{(2k_1)!\cdots (2k_N)!}.
\end{equation}
However
\[
 \frac{k_1!\cdots k_N!}{(2k_1)!\cdots (2k_N)!} \leq 1, 
\]
so an application of the Stirling's formula yields then a constant of
\[
\frac{(2j)!}{j!} = C^{2j} j^j.
\]
Noting that the bound does not depend on $N$, the desired estimate follows in the $p=2j$ case.  

To obtain \eqref{best} for  arbitrary $p \geq 2$, note first that for $p= 2 j +2$ one obtains a similar estimate with $C(2 j +2)^{1/2}$ for the same constant $C$ appearing above.  By interpolation for mixed-norm spaces, see \cite[Section 7, Theorem 2]{BP} one obtains the desired inequality \eqref{best}  for $ 2 j < p < (2j +2)$ with constant $\sqrt{2} C \sqrt{p}$, which yields the desired bound for all $p \geq 2$.
\end{proof}
We will also use the following variant of Lemma 4.5 in \cite{Tz10} to bound the probability of certain subsets of the probability space.

\begin{lemma}\label{prob_est}
Let $F$ be a real valued measurable function on a probability space $(\Omega, \mathcal{A}, \bP)$. Suppose that there exists $\alpha > 0$, $N > 0$, $k\in \bN^*$ and $C > 0$ such that for every $p\geq p_0\geq 1$ one has
\begin{align}
\|F\|_{L^p_\omega(\Omega)} \leq C N^{-\alpha} p^{\frac{k}{2}}.
\end{align}
Then, there exists $C_1$, and $c$ depending on $C$ and $p_0$ such that for $\lambda > 0$
\begin{align}
\bP( \omega \in \Omega : |F(\omega)| > \lambda) =: \bP(E_{\lambda})
 \leq C_1 e^{-c N^{\frac{2\alpha}{k}} \lambda^{\frac{2}{k}}}.
\end{align}
\end{lemma}

We now state the improved linear estimate for the zeroth iterates.
\begin{proposition}\label{prop:improved_strichartz}
Let $\phi^\omega$ be as defined in \eqref{initial_eucl}, and $u^{(0)}(x,t)$ the zeroth iterate defined in \eqref{equ:zeroth_iterate}. Let $2 \leq q,r < \infty$. Then
\[
\|\partial u^{(0)} \|_{L^q_t L^r_x(I \times \R^2)} < \infty
\]
almost surely.
\end{proposition}
\begin{proof}
We have 
\[
\partial u^{(0)}(t)=\partial W(t) \phi^\omega = \partial \cos(t|\nabla|) \phi_0^\omega + \partial \frac{\sin(t |\nabla|)}{|\nabla|}  \phi_1^\omega.
\]
We only prove the estimate for $\partial e^{\pm it |\nabla|} \phi_0^\omega$ since the other terms, including the term involving $\phi_1^\omega$, follow analogously. Let $p \geq \max(q,r)$. Then by Lemma \ref{lem:large_deviation_estimate} and Minkowski's inequality
\begin{align*}
\| \|\partial e^{\pm it |\nabla|} \phi_0^\omega \|_{L^q_t L^r_x(I \times \R^2)}\|_{L^p_\omega} &\leq \| \|\partial e^{\pm it |\nabla|} \phi_0^\omega \|_{L^p_\omega} \|_{L^q_t L^r_x(I \times \R^2)}\\
& \lesssim \sqrt{p} \biggl( \sum_{k }  \|\partial e^{\pm it |\nabla|} P_k \phi_0 \|_{L^q_t L^r_x(I \times \R^2)}^2 \biggr)^{\frac{1}{2}}.
\end{align*}
We then use H\"older's inequality and the unit-scale Bernstein estimate of Lemma \ref{bernstein} to obtain
\begin{align}
\| \|\partial e^{\pm it |\nabla|} \phi_0^\omega \|_{L^q_t L^r_x(I \times \R^2)}\|_{L^p_\omega} &\lesssim \sqrt{p} |I|^{\frac{1}{q}} \biggl( \sum_{k}  \|\partial  e^{\pm it |\nabla|} P_k \phi_0 \|_{L^\infty_t L^2_x(I \times \R^2)}^2 \biggr)^{\frac{1}{2}}\\
&\lesssim \sqrt{p} |I|^{\frac{1}{q}} \biggl( \sum_{k}  \||\nabla| P_k \phi_0 \|_{L^2_x(\R^2)}^2 \biggr)^{\frac{1}{2}},\\
&\lesssim \sqrt{p} |I|^{\frac{1}{q}}  \norm{\phi_0}_{\dot H^1_x(\R^2)},
\end{align}
and the desired result follows from Lemma \ref{prob_est} by writing
\[
\bigl\{ \|\partial u^{(0)} \|_{L^q_t L^r_x(I \times \R^2)} < \infty \bigr\} = \bigcup_{\ell =1}^\infty \bigl\{ \|\partial u^{(0)} \|_{L^q_t L^r_x(I \times \R^2)} \leq \ell \bigr\}. \qedhere
\]
\end{proof}

\section{The iteration scheme}\label{sec:scheme}
To bound the iterates, we will employ the energy estimates for the wave equation, namely
\begin{align}
\|u^{(n)} \|_{L^\infty_t \dot H^1_x} + \|\partial_t u^{(n)} \|_{L^\infty_t L^2_x}& \lesssim\|u^{(0)}\|_{L^\infty_t \dot H^1_x} + \|\partial_t u^{(0)}\|_{L^\infty_t  L^2_x} +  \|(\partial u^{(n-1)})^2\|_{L_t^1 L^2_x}\\
& = \|u^{(0)}\|_{L^\infty_t \dot H^1_x} + \|\partial_t u^{(0)}\|_{L^\infty_t L^2_x} +  \|\partial u^{(n-1)}\|^2_{L^2_t L^4_x}.
\end{align}
Hence, it suffices to obtain bounds for the term
\[
  \|\partial u^{(n-1)}\|_{L^2_t L^4_x}.
\]
To do this, we will perform an analysis based on a precise representation of the iterates, namely as a sum with products of Rademacher random variables as the coefficients. 

In the sequel, to simplify our expression for the iterate expansions, we take $\phi_1 = 0$ for our initial data. We prove a preliminary version of this representation formula in the next proposition, which we will refine subsequently. In the sequel, we will implicitly regard the indices $k_i \in \mathbb{Z}^2$ as belonging to $\mathbb{N}$ via a fixed bijection.

\begin{proposition}\label{prop:g_rep0}
Let $\{\varepsilon_k\}_{k=1}^\infty$ be the sequence of independent identically distributed Rademacher random variables used in the definition \eqref{initial_eucl}. We have the representation
\[
\partial u^{(n)}  = \sum_{j=1}^{2^n} \sum_{k_1, \ldots, k_j } \varepsilon_{k_1} \ldots \varepsilon_{k_j} G^{(n)}_{k_1, \ldots, k_j},
\]
where $k_i \in \mathbb{N}$, and for any $n \in \mathbb{N}$, and $1 \leq j \leq 2^n$, 
\begin{align}\label{equ:gn_def}
G^{(n)}_{k_{1}} := F_{k_{1}}, \quad G^{(n)}_{k_{1}, \ldots, k_{j}} := \sum_{i \in B_j} A_0\bigl(G^{(n-1)}_{k_{1}, \ldots, k_{i}}, G^{(n-1)}_{k_{i+1}, \ldots, k_{j}}  \bigr), \quad j \geq 2
\end{align}
for $F_{k_1}$ defined in \eqref{nth} and
\[
B_j = \biggl \{ i \,:\, 1 \leq i \leq j-1 \textup{ if }j \leq 2^{n-1},\,\, j - 2^{n-1} \leq i \leq 2^{n-1} \textup{ if } j > 2^{n-1}\biggr\}.
\]
\end{proposition}
\begin{remark}
We point out that $B_j$ is only defined for $j \geq 2$ and it is the collection of indices $i$ for $1 \leq i \leq 2^{n-1}$ which can contribute to the term with $j$ Rademacher random variables in the $n$-th iterate.
\end{remark}
\begin{proof}
The expression holds for $n=0$. Assume next it holds for $n-1$. We then have from the formula \eqref{nth} that
\begin{align}
\partial u^{(n)} = \partial u^{(0)} + \sum_{i =1}^{2^{n-1}} \sum_{j =1}^{2^{n-1}} \sum_{k_1, \ldots, k_i} \sum_{ \ell_1, \ldots, \ell_j} \varepsilon_{k_1} \cdots \varepsilon_{k_i} \varepsilon_{\ell_1} \cdots \varepsilon_{\ell_j}   A_0(G^{(n-1)}_{k_1, \ldots, k_i}, G^{(n-1)}_{\ell_1, \ldots, \ell_j}).
\end{align}

Now, since we are summing over all indices, all terms with $i+j$ many Rademacher random variables may be grouped together since every combination of coefficients appears in front of all of them. Consequently, we group terms with the same number of the Rademacher random variables coefficients, and perform the change of variables $i+j = j$. Then with $B_j$ as above, we obtain the result since $B_j$ contains precisely the indices $i$ which contribute to the $j$-th term.
\end{proof}

Next observe that the expression for $G^{(n)}_{k_1, \ldots, k_j}$ involves $j-1$ time integrations and using \eqref{equ:gn_def} it can be written as
\begin{align}\label{tildeg}
G^{(n)}_{k_1, \ldots, k_j} \sim \underbrace{\int \ldots \int}_{j-1 \textup{ times}} \widetilde{G}^{(n)}_{k_1, \ldots, k_j} ,
\end{align}
where
\begin{equation}\label{equ:g_def}
\begin{split}
&\widetilde{G}^{(0)}_{k_1} :=F_{k_1},\\
&\widetilde{G}^{(n)}_{k_1, \ldots, k_j} := \sum_{i \in B_j} M_{01}\bigl( \widetilde{G}^{(n-1)}_{k_1, \ldots, k_i} \cdot \widetilde{G}^{(n-1)}_{k_{i+1}, \ldots, k_j}  \bigr),\quad n\geq 1,
\end{split}
\end{equation}
with $M_{01}$ as given by \eqref{m01}. This can be seen readily by induction and the formulas above.

  We will now describe the structure of the terms in $G^{(n)}_{k_1, \ldots, k_j}$ more precisely.  We observe that these terms involve time integrals with different iterative structures. In the following proposition, we will establish that these different contributions are in bijection with a collection of full binary trees. We recall a full binary tree is a tree where each node (also called a vertex) has either no children or exactly two children.  If a node does not have a child, then it is called a leaf (or a terminal vertex).  If a node has a child, it is called an internal node. We define the height of a binary tree to be the number of edges between the root and the furthest leaf.
  
In the sequel for a binary tree $\tau$ with $j$ leaves it will be useful to set notation $\tau = \tau_i \cup \tau_{j-i}$ where $\tau_i$ and $\tau_{j-i}$ are the unique trees with $i$ and $j-i$ leaves respectively such that the root of $\tau$ has $\tau_i$ and $\tau_{j-i}$ as left and right children.

\begin{proposition}\label{trees}
Let $n\geq 0$ and $1 \leq j \leq 2^n$.   Let  $\mathcal{T}$ denote the collection of full binary trees.  Then there is an injective map from the terms appearing in $G^{(n)}_{k_1, \ldots, k_j}$, mapping terms in \eqref{equ:gn_def} to trees with $j$ leaves and $j-1$ internal nodes. We denote the image of the map by $\mathcal{T}_{j}$ and we will write
\[
G^{(n)}_{k_1, \ldots, k_j} := \sum_{\tau \in \mathcal{T}_{j}} G^{(n), \tau}_{k_1, \ldots, k_j}.
\]
\end{proposition}
\begin{remark}
Analogously to \eqref{tildeg}, we will use the notation
\begin{align}
G^{(n), \tau}_{k_1, \ldots, k_j} =  \underbrace{\int \ldots \int}_{j-1 \textup{ times}} \widetilde{G}^{(n), \tau}_{k_1, \ldots, k_j}.
\end{align}
\end{remark}

\begin{remark} Before the proof, we provide some examples.  One example arises from
\[
A_0(\partial u^{(0)}, A_0 (\partial u^{(0)}, \ldots A_0(\partial u^{(0)},\partial u^{(0)}))),
\]
and in this case, we have a contribution of 
\[
G^{(n), \tau}_{k_1, \ldots, k_j} = \underbrace{\int_0^{t_0 = t} \int_0^{t_1} \ldots \int_0^{t_{j-2}}}_{j-1 \textup{ times}} \widetilde{G}^{(n), \tau}_{k_1, \ldots, k_j}  dt_1 \ldots dt_{j-1}.
\]
This maps to a full binary tree with height $j-1$, where each left node is a leaf.  Another example arises from the contribution when $j=2^n$, and
\[
G^{(n), \tau}_{k_1, \ldots, k_j}  = \int_0^{t}  \left( \int_0^{t_1} \ldots \left(   \int_0^{t_{n-2}}  \left( \int_0^{t_{n-1} } \widetilde{G}^{(n), \tau}_{k_1, \ldots, k_j}  dt_{n} \right)^2 dt_{n-1} \right)^2 \ldots dt_2 \right)^2 dt_1.
\]
This maps to a full binary tree with height $n$, and with leaves appearing only at the final level.

\end{remark}
\begin{proof}[Proof of Proposition \ref{trees}]
We define the map inductively. Let $n\geq 0$, and $j=1$, then $G^{(n)}_{k_1}=F_{k_1}$, so there is no integral, and we map this term to a single node. For $j \geq 2$, we place a node whenever there is an appearance of the integral operator $A_0$, with the left and right children of the node corresponding to the images of the left and right terms in the bilinear operator. Working out one more example for $n=1$ and $j=2$, we have
\[
G^{(1)}_{k_{1}, k_{2}} = A_0\bigl(G^{(0)}_{k_{1}}, G^{(0)}_{k_{2}}  \bigr),
\]
which would map to a binary tree with a single internal node, and two leaves. We now proceed with the induction.

Let now $n\geq 1$, and $2\leq j \leq 2^n$, then by definition, we have
\be\label{defnGnagain}
G^{(n)}_{k_{1}, \ldots, k_{j}} = \sum_{i \in B_j} A_0\bigl(G^{(n-1)}_{k_{1}, \ldots, k_{i}}, G^{(n-1)}_{k_{i+1}, \ldots, k_{j}}  \bigr).
\ee
We take the following formula as the inductive hypothesis
\begin{align}
G^{(n)}_{k_1, \ldots, k_j} = \sum_{\tau \in \mathcal{T}_{j}} G^{(n), \tau}_{k_1, \ldots, k_j}.
\end{align}
By \eqref{defnGnagain} and the inductive hypothesis we have
\begin{align}
G^{(n+1)}_{k_{1}, \ldots, k_{j}} &= \sum_{i \in B_j} A_0\bigl(G^{(n)}_{k_{1}, \ldots, k_{i}}, G^{(n)}_{k_{i+1}, \ldots, k_{j}}  \bigr)\\
&=\sum_{i \in B_j} \sum_{\tau_{i} \in \mathcal{T}_{i}}\sum_{\tau_{j-i} \in \mathcal{T}_{j-i}}  A_0 \bigl(G^{(n), \tau_i}_{k_1, \ldots, k_i}, G^{(n), \tau_{j-i}}_{k_{i+1}, \ldots, k_j} \bigr).
\end{align}
Since $A_0$ gives an integral, we get a collection of trees where the left child comes from trees in $\mathcal{T}_{i}$, and the right child comes from trees in  $\mathcal{T}_{j-i}$ giving a tree with $i+ (j-i)=j$ leaves, and $i-1+(j-i-1)=j-1$ internal nodes.  Since the decomposition into two children trees is unique, this map is injective.
\end{proof}

\begin{corollary}
We have the representation
\begin{align}\label{equ:g_term_cor}
G^{(n), \tau}_{k_1, \ldots, k_j} =  \underbrace{\int_0^{t_0 = t} \ldots \int_0^{t_{j-2}}}_{j-1 \textup{ times}} \widetilde{G}^{(n), \tau}_{k_1, \ldots, k_j} dt_1 \ldots d_{t_{j-1}},
\end{align}
with
\be
\begin{split}
\widetilde{G}^{(0),\tau}_{k_1}& :=F_{k_1},\\
\widetilde{G}^{(n), \tau}_{k_1, \ldots, k_j} &:=   M_{01}\bigl( \widetilde{G}^{(n-1), \tau_i}_{k_1, \ldots, k_i} \cdot \widetilde{G}^{(n-1), \tau_{j-i}}_{k_{i+1}, \ldots, k_j}  \bigr),
\end{split}
\ee
where the notation $\tau_i$ and $\tau_{j-i}$ means  $\tau_i \in \mathcal{T}_i$ and $\tau_{j-i} \in \mathcal{T}_{j-i}$ for some trees $\tau_i$ and $\tau_{j-i}$ with $i$ and $j-i$ leaves respectively.

Furthermore, we observe the support of $\widetilde{G}^{(n), \tau}_{k_1, \ldots, k_j}$ is contained in the ball of radius $j$.
\end{corollary}

 Let $I_\tau(t)$ denote the $j-1$ iterated time integral which arises in \eqref{equ:g_term_cor}, that is, $I_\tau(t)$ is the number which we obtain by replacing $\widetilde{G}^{(n), \tau}_{k_1, \ldots, k_j} $ by $1$ in \eqref{equ:g_term_cor} and carrying out the time integration.

\medskip
\begin{lemma} \label{time_decay}
Let $n \geq 0$ and let $1 \leq j \leq 2^n$. For every $\tau \in \mathcal{T}_j$, we have
\[
I_\tau(t) = \frac{t^{j-1}}{C_{\tau}},
\]
and furthermore, we have the recurrence relation
\be\label{c_alpha_recurr}
\begin{split}
C_{\tau_1} &= 1, \ \ j=1, \\  C_{\tau} &= (j-1)\,C_{\tau_i} C_{\tau_{j-i}}, \,\, \tau \neq \tau_1, \ \ j\geq 2.
\end{split}
\ee
where $\tau_1$ denotes the tree with only one node.
\end{lemma}
\begin{proof}
We argue again by induction. The case $n=0$ is clear, and by definition, for $\tau = \tau_i \cup \tau_{j-i}$ we can represent $I_\tau$ as
\[
I_\tau(t) = \int_{0}^{t} dt_1 I_{\tau_i}(t_1) I_{\tau_{j-i}}(t_1),
\]
and the result follows from the definition of $C_{\tau}$ and integration.
\end{proof}

It will be useful in the sequel to introduce the notation
\begin{align}
C_{\tau, j}^* = \inf_{\tau \in \mathcal{T}_j} C_\tau,
\end{align}
and we note that the upper bound
\begin{align}\label{c_tau_star_upper}
C_{\tau, 2^n}^* \leq \prod_{k=1}^{n} (2^k - 1)^{2^{n-k}}
\end{align}
follows from considering the tree of height $n$.

The next proposition now simply follows from Proposition \ref{trees} and Lemma \ref{time_decay}.
\begin{proposition}\label{prop:g_rep}
Let $n \geq 1$ and $2 \leq j \leq 2^n$. Then
\begin{align}
\|G^{(n)}_{k_1, \ldots, k_j}\|_{L^\infty_t L^4_x} \leq \sum_{\tau \in \mathcal{T}_{j}} \frac{|I|^{j-1}}{C_\tau} \|\widetilde{G}^{(n), \tau}_{k_1, \ldots, k_j}\|_{L^\infty_t L^4_x} .
\end{align}
\end{proposition}

We next turn to establishing a suitable bound for
\begin{align}\label{equ:g_bds0}
\|\widetilde{ G}^{(n), \tau}_{k_1, \ldots, k_j}\|_{L^\infty_t L^4_x}.
\end{align}

\begin{proposition}\label{prop:g_bds2}
Let $n\geq 1$, and $2 \leq j \leq 2^n$.  There exists $C > 0$ such that for any $\tau\in \mathcal T_j$,  we have
\[
\|\widetilde{G}^{(n), \tau}_{k_1, \ldots, k_j}\|_{L^\infty_t L^4_x}  \leq C^{\frac{j-1}{2}}  \cdot \sqrt{C_\tau} \prod_{i=1}^j \|P_{k_i} \phi_0 \|_{\dot H^1_x},
\]
where $C_\tau$ is given in \eqref{c_alpha_recurr}.
\end{proposition}
\begin{proof}
We prove this by induction on $n$.   For $n=1$, by definition, Lemma \ref{bernstein}, Lemma \ref{lem:mult_bds},  and H\"older, we have
\begin{align}
\|\widetilde{G}^{(1),\tau}_{k_1, k_2}\|_{ L^4_x} &= \|M_{01}\bigl( F_{k_1} \cdot  F_{k_2} \bigr)\|_{ L^4_x} \\
& \leq C_0  \sqrt{2} \pi^\frac 14\|M_{01}\bigl( F_{k_1} \cdot  F_{k_2} \bigr)\|_{ L^2_x} \\
& \leq   C_0  \sqrt{2} \pi^\frac 14\| F_{k_1} \|_{ L^4_x}  \| F_{k_2} \|_{ L^4_x} \\
& \leq   C_0^3  \sqrt{2} \pi^\frac 34\| F_{k_1} \|_{ L^2_x}  \| F_{k_2} \|_{ L^2_x} .
\end{align}

 For general $n$, we similarly have
\begin{align}
\|\widetilde{G}^{(n),\tau}_{k_1, \ldots, k_j}\|_{ L^4_x} &\leq \|M_{01}\bigl( \widetilde{G}^{(n-1), \tau_i}_{k_1, \ldots, k_i} \cdot \widetilde{G}^{(n-1), \tau_{j-i}}_{k_{i+1}, \ldots, k_j}  \bigr)\|_{ L^4_x} \\
& \leq C_0 \pi^\frac 14 j^\frac 12 \|M_{01}\bigl(  \widetilde{G}^{(n-1), \tau_i}_{k_1, \ldots, k_i} \cdot \widetilde{G}^{(n-1), \tau_{j-i}}_{k_{i+1}, \ldots, k_j}  \bigr)\|_{ L^2_x} \\
& \leq   C_0 \pi^\frac 14 j^\frac 12 \| \widetilde{G}^{(n-1), \tau_i}_{k_1, \ldots, k_i} \|_{L^\infty_t L^4_x}  \cdot \|\widetilde{G}^{(n-1), \tau_{i-j}}_{k_{i+1}, \ldots, k_j} \|_{ L^4_x} .
\end{align}

We now let $ C > 0$ be such that 
\be\label{eq_C}
C_0^2\pi^\frac 12 j \leq C (j-1)
\ee
for all $j \geq 2$. Hence, using the inductive hypothesis and \eqref{eq_C}, we obtain
\begin{align}
\|\widetilde{G}^{(n),\tau}_{k_1, \ldots, k_j}\|_{ L^4_x} & \leq  C_0 \pi^\frac 14 j^\frac 12  C^{\frac{i-1}{2} }\sqrt{C_{\tau_{i} }} \prod_{\ell=1}^i \|P_{k_\ell} \phi_0 \|_{\dot H^1_x}  C^{ \frac{j-i-1}{2}} \sqrt{ C_{\tau_{j-i}}} \prod_{\ell= i + 1}^j \|P_{k_{\ell}} \phi_0 \|_{\dot H^1_x}\\
& \leq C^{\frac{j-2}{2}}\sqrt{ C(j-1) C_{\tau_{i}} C_{\tau_{j-i}}    } \prod_{\ell= i}^j \|P_{k_{\ell}} \phi_0 \|_{\dot H^1_x},
\end{align}
which then yields the result by \eqref{c_alpha_recurr}.
\end{proof}

\begin{remark} \label{bern_dim}
The factor from Bernstein's inequality will be $C_0 j^{\frac{3}{4}}$ in dimension three and $C_0 j$ in dimension four, in which case we will obtain a modified bound for Proposition \ref{prop:g_bds2}, specifically the power of $C_\tau$ will be $3/4$ in dimension three and $1$ in dimension four.
\end{remark}
We conclude this section by stating the required $L^2_t L^4_x$ bounds we will rely on in the proof of the main theorem. 
\begin{proposition}
For any $n \geq 0$, $1\leq j \leq 2^n$ and $G^{(n)}_{k_1, \ldots, k_j}$ as above, we have
\begin{align} \label{equ:necessary_g_bds}
\| G^{(n)}_{k_1, \ldots, k_j}\|_{ L^2_t L^4_x} \leq |I|^{j - \frac{1}{2}} C^{\frac{j-1}{2}}\frac{1}{\sqrt{C_{\tau, j}^*}} \prod_{i=1}^{j}  \|P_{k_i} \phi_0 \|_{\dot H^1_x}.
\end{align}
\end{proposition}
\begin{proof}
For $j=1$, this holds by Proposition \ref{prop:improved_strichartz}, while for $j \geq 2$, this result is a summary of the previous bounds, together with the fact that the number of full binary trees with $j$ leaves is given by the $j$-th Catalan number, which is exponential in $j$, and hence can be absorbed into the $C^{\frac{j-1}{2}}$ factor.
\end{proof}

\section{Proof of main probabilistic bounds}\label{sec:prob_bds}
By the discussion at the beginning of Section \ref{sec:scheme}, and in view of Proposition \ref{prop:g_rep0}, the goal of this section is to establish our main probabilistic bounds for the expression
\[
\biggl\| \biggl\| \sum_{k_1, \ldots, k_j} \varepsilon_{k_1} \ldots \varepsilon_{k_j} G^{(n)}_{k_1, \ldots, k_j}  \biggr\|_{L^2_t L^4_x}\biggr\|_{L^p_\omega}.
\]
This will suffice for establishing bounds on the iterates. 

Provided $p \geq 4$, we can use Minkowski's inequality to bring the $L^p_\omega$ norm inside, and  hence we first consider 
\begin{align}
\biggl\| \sum_{k_1, \ldots, k_j} \varepsilon_{k_1} \ldots \varepsilon_{k_j} G^{(n)}_{k_1, \ldots, k_j}  \biggr\|_{L^p_\omega}  .
\end{align}
We group the summation over $k_1, \ldots, k_j$ based on how many distinct indices $k_i$ appear, ranging from $r=1, \ldots, j$.  So for example,  when $r=j$, all the indices are different, while if $r=1$, then all the indices are the same, and the product of the random variables reduces to $\varepsilon^j_k$.  Moreover, we observe that for each $r$, we ask in how many ways we can distribute $j$ (labeled) indices into $r$ (unlabeled) groups.  Stirling numbers of the second kind provide an answer to this and are labeled by $S(j,r)$ (more on $S(j,r)$ below).  We let $P_{j,r}$ denote the collection of such distributions.  An element of $P_{j,r}$ can be identified with $\vec{k}$ to denote a vector of length $j$, which has $r$ distinct indices $k_i$ appearing $\alpha_i$ times so that $\alpha_1+\alpha_2+\cdots+\alpha_r=j$.  

Ultimately we need to estimate
\begin{align}\label{our_est}
\biggl\| \sum_{k_1, \ldots, k_j} \varepsilon_{k_1} \ldots \varepsilon_{k_j} G^{(n)}_{k_1, \ldots, k_j}  \biggr\|_{L^p_\omega}  
& \leq \sum_{r=1}^j  \sum_{ \vec{k} \in P_{j,r}} \biggl\|  \sum_{k_1, \ldots,  k_r, k_i \neq k_\ell}   \varepsilon_{k_1}^{\alpha_1} \ldots \varepsilon_{k_r}^{\alpha_r}   G^{(n)}_{\vec{k} }  \biggr\|_{L^p_\omega}.
\end{align}

First we consider one of the terms
\begin{align}\label{sumtoconsider}
\biggl\|  \sum_{k_1, \ldots,  k_r, k_i \neq k_\ell}   \varepsilon_{k_1}^{\alpha_1} \ldots \varepsilon_{k_r}^{\alpha_r}   G^{(n)}_{\vec{k} }  \biggr\|_{L^p_\omega}.
\end{align}
We would like to apply Lemma \ref{blemma} to estimate this expression. However, we are not quite yet in a suitable context since the same random variable may appear in more than one summand. Thus we follow an argument used in \cite{ErdosYau}. 

We fix a large $N \in \mathbb{N}$ and we let $\mathbb{N}_N := \{1, \ldots, N\}$. We will estimate
\[
 \sum_{k_1, \ldots,  k_r, k_i \neq k_\ell,\, k_i \in \mathbb{N}_N}   \varepsilon_{k_1}^{\alpha_1} \ldots \varepsilon_{k_r}^{\alpha_r}   G^{(n)}_{\vec{k} }
\]
uniformly in $N$ which will enable us to conclude the desired bound for \eqref{sumtoconsider}. We use the identity
\begin{align}\label{partition}
1 = \frac{1}{r^{N-r}} \sum_{I_1 \sqcup \ldots \sqcup I_r = \mathbb{N}_N} \mathbf{1}(k_1 \in I_1) \cdots \mathbf{1}(k_r \in I_r),
\end{align}
where $I \sqcup J$ denote a disjoint union, and the sum is over all such disjoint unions of $r$-many arbitrary subsets $I_1, \ldots, I_r$ of $\mathbb{N}_N$. Since the summand is zero if any of the $I_i = \varnothing$, we will assume that the $I_i$ are non-empty. We note that this identity holds for all tuples $(k_1, \ldots, k_r)$ with $k_i \in \mathbb{N}_N$ such that $k_i \neq k_\ell$ for $i \neq \ell$. 

The normalization factor in \eqref{partition} is the number of distinct disjoint unions of $\mathbb{N}_N$ a fixed tuple $(k_1, \ldots, k_r)$ with $k_i \neq k_\ell$ for $i \neq \ell$, may appear in. We compute it as follows: we need to allocate the points
\[
\{1, \ldots, N\} \setminus \{k_1, \ldots, k_r\}
\]
to sets $I_1 \ni k_1, \ldots, I_r \ni k_r$. There are $r$-many options for each of the remaining $N-r$ points, which yields the normalization factor.

Consequently, using \eqref{partition} we can write 
\begin{align}
& \sum_{k_1, \ldots,  k_r, k_i \neq k_\ell, |k_i| \leq N}   \varepsilon_{k_1}^{\alpha_1} \ldots \varepsilon_{k_r}^{\alpha_r}   G^{(n)}_{\vec{k} } \\
 & = \frac{1}{r^{N-r}}  \sum_{I_1 \sqcup \ldots \sqcup I_r= \mathbb{N}_N} \sum_{k_1 \in I_1, \ldots,  k_r \in I_r}   \varepsilon_{k_1}^{\alpha_1} \ldots \varepsilon_{k_r}^{\alpha_r}   G^{(n)}_{\vec{k} } .
 \end{align}
We first tackle the case where each $\alpha_i$ is odd. We may now use Lemma \ref{blemma} and we define
\[
b_{k_1} =  \sum_{k_2 \in I_2, \ldots,  k_r \in I_r}   \varepsilon_{k_2}^{\alpha_2} \ldots \varepsilon_{k_r}^{\alpha_r}   G^{(n)}_{\vec{k} }.
\]
and we write
\begin{align}
\biggl\| \sum_{k_1 \in I_1, \ldots,  k_r \in I_r}   \varepsilon_{k_1}^{\alpha_1} \ldots \varepsilon_{k_r}^{\alpha_r}   G^{(n)}_{\vec{k} } \biggr\|_{L^p_\omega}  = \biggl\| \sum_{k_1 \in I_1} \varepsilon_{k_1}^{\alpha_1} b_{k_1} \biggr\|_{L^p_\omega}.
\end{align}
By construction, it is now the case that the family $\{b_{k_1}\}$ is independent of the family $\{\varepsilon_{k_1}\}$,  and applying Lemma \ref{blemma}, we obtain
\begin{align}
\bigl\|  \sum_{k_1 \in I_1}   \varepsilon_{k_1}^{\alpha_1}  b_{k_1}  \bigr\|_{L^p_\omega} &\leq C \sqrt{p} \biggl\| \biggl( \sum_{k_1 \in I_1}   |b_{k_1}|^2 \biggr)^{\frac{1}{2}}  \biggr\|_{L^p_\omega}\\
& \leq  C \sqrt{p} \biggl(\sum_{k_1 \in I_1}   \|b_{k_1}\|_{L^{p}_\omega}^2 \biggr)^{\frac{1}{2}}.
\end{align}
We iterate this process and eventually we get
\begin{align}
\biggl\| \sum_{k_1 \in I_1, \ldots,  k_r \in I_r}   \varepsilon_{k_1}^{\alpha_1} \ldots \varepsilon_{k_r}^{\alpha_r}   G^{(n)}_{\vec{k} } \biggr\|_{L^p_\omega} \leq C^r p^{\frac{r}{2}} \biggl(\sum_{k_1 \in I_1, \ldots,  k_r \in I_r}  \bigl|     G^{(n)}_{\vec{k} }  \bigr|^2 \biggr)^{\frac{1}{2}} .
\end{align}
In the general case, where the $\alpha_i$ are not all odd, we order the indices and apply the triangle inequality for the even indices. More precisely, if there are even powers $\alpha_i$, then $\varepsilon_{k_i}^{\alpha_i} = 1$, and in that case, we use the triangle inequality at that step in the estimates, for instance
\begin{align}
\bigl\|  \sum_{k_1 \in I_1}   \varepsilon_{k_1}^{\alpha_1}  b_{k_1}  \bigr\|_{L^p_\omega} \leq   \sum_{k_1 \in I_1}   \|  b_{k_1} \|_{L^p_\omega}.
\end{align}
Thus, reordering the $k_i$ and letting $r_{\textup{o}}$ denote the number of odd $\alpha_i$, we obtain
\begin{align}
\bigl\|  \sum_{k_1 \in I_1}   \varepsilon_{k_1}^{\alpha_1}  b_{k_1}  \bigr\|_{L^p_\omega} \leq \sum_{k_i \in I_i, \alpha_i \textup{ even}} C^{r_{\textup{o}}} p^{\frac{r_{\textup{o}}}{2}}  \biggl(\sum_{k_i \in I_i, \alpha_i \textup{ odd}}  \bigl|   G^{(n)}_{\vec{k} }  \bigr|^2 \biggr)^{\frac{1}{2}}.
\end{align}
We put everything together and we have
\begin{align}
& \sum_{r=1}^j \sum_{ \vec{k} \in P_{j,r}} \biggl\| \sum_{k_1, \ldots,  k_r, k_i \neq k_\ell, |k_i| \leq N}   \varepsilon_{k_1}^{\alpha_1} \ldots \varepsilon_{k_r}^{\alpha_r}   G^{(n)}_{\vec{k} } \biggr\|_{L^p_\omega} \\
 & \leq  \sum_{r=1}^j \sum_{ \vec{k} \in P_{j,r}} \frac{1}{r^{N-r}}  \sum_{I_1 \sqcup \ldots \sqcup I_r= \mathbb{N}_N} \sum_{k_i \in I_i, \alpha_i \textup{ even}} C^{r_{\textup{o}}} p^{\frac{r_{\textup{o}}}{2}}  \biggl(\sum_{k_i \in I_i, \alpha_i \textup{ odd}}  \bigl|   G^{(n)}_{\vec{k} }  \bigr|^2 \biggr)^{\frac{1}{2}}.
 \end{align}
 Now we estimate the expression
 \begin{align}
 \frac{1}{r^{N-r}}  \sum_{I_1 \sqcup \ldots \sqcup I_r= \mathbb{N}_N} 1.
 \end{align}
 The sum will count the number of surjective functions on $N$ items into $r$ items, and by inclusion exclusion this is equal to 
 \[
 r! S(N,r),
 \]
 where $S(N,r)$ is the Stirling number of the second kind, given by the formula
\begin{align}\label{equ:stirling_second}
S(N,r) = \frac{1}{r!} \sum_{j=0}^r (-1)^{r-j} \binom{r}{j} j^N.
\end{align}
We can use the trivial bound
\[
r! S(N, r) \leq r^N,
\]
which is the total number of functions on $N$ items into $r$ items, and hence we obtain that
 \begin{align} \label{surj_function_asymptotics}
 \frac{1}{r^{N-r}}  \sum_{I_1 \sqcup \ldots \sqcup I_r= \mathbb{N}_N} 1 = \frac{r! S(N,r)}{r^{N-r}} \leq r^r.
 \end{align}
Thus we obtain a uniform bound in $N$, which yields the desired bound
 \begin{align} \label{equ:where_we_are}
\biggl\| \sum_{k_1, \ldots, k_j} \varepsilon_{k_1} \ldots \varepsilon_{k_j} G^{(n)}_{k_1, \ldots, k_j}  \biggr\|_{L^p_\omega}   & \leq  \sum_{r=1}^j \sum_{ \vec{k} \in P_{j,r}} \sum_{k_i, \alpha_i \textup{ even}}  C^{r_{\textup{o}}} p^{\frac{r_{\textup{o}}}{2}}  r^r\biggl(\sum_{k_i, \alpha_i \textup{ odd}}  \bigl|    G^{(n)}_{\vec{k} }  \bigr|^2 \biggr)^{\frac{1}{2}}.
\end{align}

We are now ready to establish the relevant bound for \eqref{our_est} in the following proposition.
\begin{proposition}\label{prop:iterate_strichartz_bds}
Let $p\geq 4$, and  let $G^{(n)}_{k_1, \ldots, k_j}$ be given as above. Then
\begin{align}
\biggl\| \bigl\| \sum_{k_1, \ldots, k_j} \varepsilon_{k_1} \ldots \varepsilon_{k_j} G^{(n)}_{k_1, \ldots, k_j}  \bigr\|_{L^2_t L^4_x}\biggr\|_{L^p_\omega} \leq C^j  \|\phi_0\|_{\dot H_x^1}^j |I|^{j-\frac{1}{2}} p^{\frac{j}{2}} \frac{j!}{\sqrt{C_{\tau, j}^*}}
\end{align}
for a constant $C$ independent of $n$.
\end{proposition}

\begin{proof}
In this proof, we will use $C$ to denote explicit constants whose values do not depend on the various parameters. We allow the value of $C$ to change line to line. 

By Minkowski's inequality we first have
\begin{align}
\biggl\| \bigl\| \sum_{k_1, \ldots, k_j} \varepsilon_{k_1} \ldots \varepsilon_{k_j} G^{(n)}_{k_1, \ldots, k_j}  \bigr\|_{L^2_t L^4_x}\biggr\|_{L^p_\omega} 
\leq  \biggl\| \bigl\| \sum_{k_1, \ldots, k_j} \varepsilon_{k_1} \ldots \varepsilon_{k_j} G^{(n)}_{k_1, \ldots, k_j} \bigr\|_{L^p_\omega}  \biggr\|_{L^2_t L^4_x},
\end{align}
so we can take the $L^2_t L^4_x$ norm of \eqref{equ:where_we_are} and use \eqref{equ:necessary_g_bds} to obtain
\begin{align}
 & \sum_{r=1}^j  \sum_{ \vec{k} \in P_{j,r}} \sum_{k_i, \alpha_i \textup{ even}} C^{r_{\textup{o}}} p^{\frac{r_{\textup{o}}}{2}} r^r \biggl(\sum_{k_i, \alpha_i \textup{ odd}}  \bigl\|   G^{(n)}_{\vec{k} }  \bigr\|_{L^2_tL^4_x}^2 \biggr)^{\frac{1}{2}}\\
 & \leq  \frac{|I|^{j-\frac{1}{2}} C^{\frac{j-1}{2}}}{\sqrt{C_{\tau, j}^*}} \sum_{r=1}^j  C^r p^{\frac{r}{2}} r^r\sum_{ \vec{k} \in P_{j,r}} \sum_{k_i, \alpha_i \textup{ even}}   \biggl(\sum_{k_i, \alpha_i \textup{ odd}}  \biggl(   \prod_{i=1}^r \|P_{k_i} \phi_0 \|^{\alpha_i}_{\dot H^1_x} \biggr)^2 \biggr)^{\frac{1}{2}}\\
  & \leq \frac{|I|^{j-\frac{1}{2}} C^{\frac{j-1}{2}}}{\sqrt{C_{\tau, j}^*}}  C^j p^{\frac{j}{2}}\sum_{r=1}^j r^r \sum_{ \vec{k} \in P_{j,r}} \sum_{k_i, \alpha_i \textup{ even}}  
  \biggl(\sum_{k_i, \alpha_i \textup{ odd}}     \prod_{i=1}^r \|P_{k_i} \phi_0 \|^{2\alpha_i}_{\dot H^1_x} \biggr)^{\frac{1}{2}},
\end{align}
where we used $r_{\textup{o}}\leq r\leq j$.  Next we observe that 
\be
\sum_{k_i, \alpha_i \textup{ even}}  
  \biggl(\sum_{k_i, \alpha_i \textup{ odd}}     \prod_{i=1}^r \|P_{k_i} \phi_0 \|^{2\alpha_i}_{\dot H^1_x} \biggr)^{\frac{1}{2}}\leq\|\phi_0\|_{\dot H_x^1}^j.
\ee
We are then left with estimating
\[
\|\phi_0\|_{\dot H_x^1}^j \frac{|I|^{j-\frac{1}{2}} C^{\frac{j-1}{2}}}{\sqrt{C_{\tau, j}^*}} C^j p^{\frac{j}{2}} \sum_{r=1}^j r^r\sum_{ \vec{k} \in P_{j,r}} 1.
\]
Since there are $S(j,r)$ many vectors in $P_{j,r}$ we use a refined upper bound for the Stirling numbers of the second kind (c.f. \cite[Theorem 3]{RD69}):
\[
S(j,r) \leq \frac{1}{2} {j \choose r} r^{j-r}.
\]
Together with the upper bound for the binomial coefficients
\[
{j \choose r} \leq \left( \frac{e j}{r} \right)^r,
\]
this yields
\begin{align}
\|\phi_0\|_{\dot H_x^1}^j \frac{|I|^{j-\frac{1}{2}} C^{\frac{j-1}{2}}}{\sqrt{C_{\tau, j}^*}} C^j p^{\frac{j}{2}}  \sum_{r=1}^j \frac{1}{2}  \left( \frac{e j}{r} \right)^r r^{j} &=  \|\phi_0\|_{\dot H_x^1}^j \frac{|I|^{j-\frac{1}{2}} C^{\frac{j-1}{2}}}{\sqrt{C_{\tau, j}^*}} C^j p^{\frac{j}{2}}  \sum_{r=1}^j \frac{1}{2} e^r j^r r^{j -r} \\
& \leq  \|\phi_0\|_{\dot H_x^1}^j \frac{|I|^{j-\frac{1}{2}} C^{\frac{j-1}{2}}}{\sqrt{C_{\tau, j}^*}} C^j p^{\frac{j}{2}} j^j  \sum_{r=1}^j \frac{1}{2} e^r.
\end{align}
Ultimately, we obtain that there exists some constant $C$ (possibly different from above, but still uniform in $n$) such that 
\begin{equation} \label{final_strichartz_ests}
\begin{split}
\biggl\| \bigl\|  \sum_{k_1, \ldots, k_j} \varepsilon_{k_1} \ldots \varepsilon_{k_j} G^{(n)}_{k_1, \ldots, k_j} \bigr\|_{L^2_t L^4_x} \biggr\|_{L^p_\omega} 
 \leq  C^j  \|\phi_0\|_{\dot H_x^1}^j |I|^{j-\frac{1}{2}} p^{\frac{j}{2}} \frac{j!}{\sqrt{C_{\tau, j}^*}}.
\end{split}
\end{equation}
This yields the desired bound and concludes the proof.
\end{proof}
\begin{corollary}\label{summary}
Let $n\geq0$. Then
\begin{align}
\bigl\|\bigl\| \partial u^{(n)} \bigr\|_{L^2_t L^4_x} \bigr\|_{L^p_\omega} \leq  \sum_{j=1}^{2^n} C^j  \|\phi_0\|_{\dot H_x^1}^j |I|^{j-\frac{1}{2}} p^{\frac{j}{2}} j!
\end{align}
for a constant $C$ independent of $n$.
\end{corollary}
\begin{proof}
This follows by Proposition \ref{prop:g_rep0} and the trivial lower bound $C_{\tau, j}^* \geq 1$.
\end{proof}
\begin{remark}
We note that using the upper bound in \eqref{c_tau_star_upper}, we have that
\begin{align}
C_{\tau, 2^n}^* \leq \prod_{k=1}^{n} 2^{k 2^{n-k}}.
\end{align}
Now we can compute exactly that
\begin{align}
\sum_{k=1}^{n} k 2^{n-k} =2^{n + 1} -n - 2,
\end{align}
hence we see the factor of $C_{\tau,j}^*$ in our bound \eqref{final_strichartz_ests} does not suffice to cancel the factorial growth.
\end{remark}

\section{Proof of Theorem \protect{\ref{main}}} \label{sec:proof_main}
The main theorem will follow from the following result and Lemma \ref{prob_est}.

\begin{proposition}\label{prop:finite_nsets}
There exist constants $C > 0$ and $\delta > 0$ such that for every $n \in \mathbb{N}$, and every interval $I$ with $|I| < \delta$ the following holds
\begin{align}
\biggl\|\bigl\| \partial u^{(n)} \bigr\|_{L^2_t L^4_x(I \times \mathbb{R}^2)} \biggr\|_{L^p_\omega} & \leq C p^{\frac{2^n}{2}}  \|\phi_0\|_{\dot H_x^1}  |I|^{\frac{1}{2}}(2^n)!
\end{align}
\end{proposition}
\begin{proof}
By Corollary \ref{summary}
\begin{align}
\biggl\|\bigl\| \partial u^{(n)} \bigr\|_{L^2_t L^4_x} \biggr\|_{L^p_\omega} \leq  \sum_{j=1}^{2^n} C^j  \|\phi_0\|_{\dot H_x^1}^j |I|^{j-\frac{1}{2}} p^{\frac{j}{2}} j!.
\end{align}

We rewrite this bound as
\begin{align}
& C  \|\phi_0\|_{\dot H_x^1} |I|^{\frac{1}{2}} \sum_{j=1}^{2^n} C^{j -1}  \|\phi_0\|_{\dot H_x^1}^{j-1} |I|^{j-1} p^{\frac{j}{2}} j!\\
 & \leq  C  \|\phi_0\|_{\dot H_x^1} |I|^{\frac{1}{2}}p^{\frac{2^n}{2}}   (2^n)!\sum_{j=0}^{2^n-1} C^{j }  \|\phi_0\|_{\dot H_x^1}^{j} |I|^{j},
\end{align}
and we note that the sum in $j$ is bounded by one, say, provided we choose $I \subseteq \R$ such that
\[
C \|\phi_0\|_{\dot H_x^1} |I| < \frac{1}{2}. \qedhere
\]
\end{proof}

\begin{corollary}\label{cor:full_meas}
There exist $C, c > 0$, independent of $n$, such that for any $I \subseteq \R$ with $|I|$ sufficiently small, we have
\[
\mathbb{P} \bigl\{ \bigl\| \partial u^{(n)} \bigr\|_{L^2_t L^4_x(I \times \mathbb{R}^2)} > \lambda \bigr\} \leq C \exp \left(- 
\frac{c\lambda^{\frac{1}{2^{n-1}}}} { (2^n \|\phi_0\|_{\dot H_x^1} |I|^\frac 12)^{\frac{1}{2^{n-1}} }}\right).
\]
In particular, for any fixed $n$,  we have
\[
\mathbb{P} \bigl\{ \bigl\| \partial u^{(n)} \bigr\|_{L^2_t L^4_x(I \times \mathbb{R}^2)} < \infty \bigr\} = 1.
\]
\end{corollary}
\begin{proof}
This is an application of Lemma \ref{prob_est} to the result of Proposition \ref{prop:finite_nsets}  with $k=2^n$, $\alpha=1$.
\end{proof}

The proof of our main theorem now follows readily from these estimates.

\begin{proof}[Proof of Theorem \protect{\ref{main}}]
We fix $(\phi_0, \phi_1)$ and $T > 0$  so that $T$ is sufficiently small in the sense required for Corollary \ref{cor:full_meas}. By the energy estimates, we recall that we have
\begin{align}\label{equ:energy}
\|u^{(n)} \|_{L^\infty_t \dot H^1_x} + \|\partial_t u^{(n)} \|_{L^\infty_t L^2_x}  \lesssim\|u^{(0)}\|_{L^\infty_t \dot H^1_x} + \|\partial_t u^{(0)}\|_{L^\infty_t  L^2_x}  +   \|\partial u^{(n-1)}\|^2_{L^2 L^4_x}.
\end{align}
Since $\phi^\omega$ almost surely in $\dot H^1(\R^2) \times L^2_x(\R^2)$, we let $\Sigma$ be such that for $\omega \in \Sigma$,
\[
\|\phi_0^\omega \|_{L^\infty_x \dot H^1_x} + \| \phi_1^\omega \|_{L^\infty_x  L^2_x}  < \infty,
\]
which then provides a bound for 
\[
\|u^{(0)}\|_{L^\infty_t \dot H^1_x} + \|\partial_t u^{(0)}\|_{L^\infty_t  L^2_x}.
\]
Now, for $n \geq 0$, we let
\[
\Sigma_n = \bigl\{ \omega \,:\,\bigl\| \partial u^{(n)} \bigr\|_{L^2_t L^4_x(I \times \mathbb{R}^2)}  < \infty\bigr\},
\]
then $\Sigma_n$ has full measure by Proposition \ref{prop:improved_strichartz} and  Corollary \ref{cor:full_meas}. We then set
\[
\Sigma_T := \Sigma \cap \bigcap_{n=0}^\infty \Sigma_n,
\]
and we observe that $\mathbb{P}(\Sigma_T) = 1$ since it is the countable intersection of sets of full measure. 

Now, for any $\omega \in\Sigma_T$, and any $n \geq 0$, we have by \eqref{equ:energy} that 
\[
\bigl(u^{(n)}, \partial_t u^{(n)}\bigr) \in L^\infty_t \bigl([0,T]; \dot H^1_x(\R^2) \times L^2_x(\R^2)\bigr).
\]
The continuity follows from the definition of the iterates. This completes the proof.
\end{proof}
  
\bibliographystyle{myamsplain}
\bibliography{ref}

\def\cprime{$'$} \def\cprime{$'$} \newcommand{\SortNoop}[1]{}
\providecommand{\bysame}{\leavevmode\hbox to3em{\hrulefill}\thinspace}
\providecommand{\MR}{\relax\ifhmode\unskip\space\fi MR }
\providecommand{\MRhref}[2]{%
  \href{http://www.ams.org/mathscinet-getitem?mr=#1}{#2}
}
\providecommand{\href}[2]{#2}
\begin{thebibliography}{10}

\bibitem{BP}
Benedek, A. and Panzone, R., \emph{The space {$L^{p}$}, with mixed norm}, Duke
  Math. J. \textbf{28} (1961), 301--324. \MR{0126155}

\bibitem{BOP1}
B\'enyi, A., Oh, T., and Pocovnicu, O., \emph{{Wiener randomization on
  unbounded domains and an application to almost sure well-posedness of NLS}},
  arXiv:1405.7326.

\bibitem{BOP}
{B{\'e}nyi}, {\'A}., {Oh}, T., and {Pocovnicu}, O., \emph{{Higher order
  expansions for the probabilistic local Cauchy theory of the cubic nonlinear
  Schr$\backslash$``odinger equation on $\mathbb{R}^3$}}, ArXiv e-prints
  (2017).

\bibitem{B94}
Bourgain, J., \emph{Periodic nonlinear {S}chr\"odinger equation and invariant
  measures}, Comm. Math. Phys. \textbf{166} (1994), no.~1, 1--26.

\bibitem{B96}
\bysame, \emph{Invariant measures for the {$2$}{D}-defocusing nonlinear
  {S}chr\"odinger equation}, Comm. Math. Phys. \textbf{176} (1996), no.~2,
  421--445.

\bibitem{BGP}
\bysame, \emph{Invariant measures for the {G}ross-{P}iatevskii equation}, J.
  Math. Pures Appl. (9) \textbf{76} (1997), no.~8, 649--702.

\bibitem{B94Z}
Bourgain, J., \emph{On the {C}auchy and invariant measure problem for the
  periodic {Z}akharov system}, Duke Math. J. \textbf{76} (1994), no.~1,
  175--202. \MR{1301190}

\bibitem{BournaveasCandy_ME}
Bournaveas, N. and Candy, T., \emph{Local well-posedness for the space-time
  monopole equation in {L}orenz gauge}, NoDEA Nonlinear Differential Equations
  Appl. \textbf{19} (2012), no.~1, 67--78. \MR{2885552}

\bibitem{BT3}
Burq, N. and Tzvetkov, N., \emph{Invariant measure for a three dimensional
  nonlinear wave equation}, Int. Math. Res. Not. (2007), no.~22.

\bibitem{BT1}
\bysame, \emph{Random data {C}auchy theory for supercritical wave equations.
  {I}. {L}ocal theory}, Invent. Math. \textbf{173} (2008), no.~3, 449--475.

\bibitem{BT2}
\bysame, \emph{Random data {C}auchy theory for supercritical wave equations.
  {II}. {A} global existence result}, Invent. Math. \textbf{173} (2008), no.~3,
  477--496.

\bibitem{BT4}
\bysame, \emph{Probabilistic well-posedness for the cubic wave equation}, J.
  Eur. Math. Soc. \textbf{16} (2014), no.~1, 1--30.

\bibitem{Christodoulou86}
Christodoulou, D., \emph{Global solutions of nonlinear hyperbolic equations for
  small initial data}, Comm. Pure Appl. Math. \textbf{39} (1986), no.~2,
  267--282. \MR{820070}

\bibitem{CKLS2}
C\^ote, R., Kenig, C.~E., Lawrie, A., and Schlag, W., \emph{Characterization of
  large energy solutions of the equivariant wave map problem: {II}}, Amer. J.
  Math. \textbf{137} (2015), no.~1, 209--250. \MR{3318090}

\bibitem{CzubakT}
Czubak, M., Ph.D. Thesis, University of Texas at Austin (2008).

\bibitem{CzubakME}
\bysame, \emph{Local wellposedness for the {$2+1$}-dimensional monopole
  equation}, Anal. PDE \textbf{3} (2010), no.~2, 151--174. \MR{2657452}

\bibitem{CP}
Czubak, M. and Pikula, N., \emph{Low regularity well-posedness for the 2{D}
  {M}axwell-{K}lein-{G}ordon equation in the {C}oulomb gauge}, Commun. Pure
  Appl. Anal. \textbf{13} (2014), no.~4, 1669--1683. \MR{3177753}

\bibitem{DaiTerng}
Dai, B. and Terng, C.-L., \emph{B\"acklund transformations, {W}ard solitons,
  and unitons}, J. Differential Geom. \textbf{75} (2007), no.~1, 57--108.
  \MR{2282725}

\bibitem{DCU}
Dai, B., Terng, C.-L., and Uhlenbeck, K., \emph{On the space-time monopole
  equation}, Surveys in differential geometry. {V}ol. {X}, Surv. Differ. Geom.,
  vol.~10, Int. Press, Somerville, MA, 2006, pp.~1--30. \MR{MR2408220}

\bibitem{DFS_10}
D'Ancona, P., Foschi, D., and Selberg, S., \emph{Product estimates for
  wave-{S}obolev spaces in {$2+1$} and {$1+1$} dimensions}, Nonlinear partial
  differential equations and hyperbolic wave phenomena, Contemp. Math., vol.
  526, Amer. Math. Soc., Providence, RI, 2010, pp.~125--150. \MR{2731990}

\bibitem{DLuM1}
Dodson, B., {L\"uhrmann, J.}, and Mendelson, D., \emph{{Almost sure scattering
  for the 4D energy-critical defocusing nonlinear wave equation with radial
  data}}, arXiv:1703.09655.

\bibitem{ErdosYau}
Erdos, L. and Yau, H.-T., \emph{A dynamical approach to random matrix theory},
  Courant Lecture Notes in Mathematics, vol.~28, Courant Institute of
  Mathematical Sciences, New York; American Mathematical Society, Providence,
  RI, 2017. \MR{3699468}

\bibitem{FK}
Foschi, D. and Klainerman, S., \emph{Bilinear space-time estimates for
  homogeneous wave equations}, Ann. Sci. \'Ecole Norm. Sup. (4) \textbf{33}
  (2000), no.~2, 211--274. \MR{MR1755116 (2001g:35145)}

\bibitem{FreedUhl}
Freed, D.~S. and Uhlenbeck, K.~K., \emph{Instantons and four-manifolds},
  Mathematical Sciences Research Institute Publications, vol.~1,
  Springer-Verlag, New York, 1984. \MR{757358}

\bibitem{GlimmJaffe}
Glimm, J. and Jaffe, A., \emph{Quantum physics}, second ed., Springer-Verlag,
  New York, 1987, A functional integral point of view. \MR{887102}

\bibitem{GrigoryanTanguay}
{Grigoryan}, V. and {Tanguay}, A., \emph{{Improved well-posedness for the
  quadratic derivative nonlinear wave equation in 2D}}, ArXiv e-prints (2013).

\bibitem{GN}
Grigoryan, V. and Nahmod, A.~R., \emph{Almost critical well-posedness for
  nonlinear wave equations with {$Q_{\mu\nu}$} null forms in 2{D}}, Math. Res.
  Lett. \textbf{21} (2014), no.~2, 313--332. \MR{3247059}

\bibitem{Grunrock}
Gr\"unrock, A., \emph{On the wave equation with quadratic nonlinearities in
  three space dimensions}, J. Hyperbolic Differ. Equ. \textbf{8} (2011), no.~1,
  1--8. \MR{2796047}

\bibitem{Klainerman84}
Klainerman, S., \emph{The null condition and global existence to nonlinear wave
  equations}, Nonlinear systems of partial differential equations in applied
  mathematics, {P}art 1 ({S}anta {F}e, {N}.{M}., 1984), Lectures in Appl.
  Math., vol.~23, Amer. Math. Soc., Providence, RI, 1986, pp.~293--326.
  \MR{837683 (87h:35217)}

\bibitem{KlainermanMachedon93}
Klainerman, S. and Machedon, M., \emph{Space-time estimates for null forms and
  the local existence theorem}, Comm. Pure Appl. Math. \textbf{46} (1993),
  no.~9, 1221--1268. \MR{MR1231427 (94h:35137)}

\bibitem{KMYM}
\bysame, \emph{Finite energy solutions of the {Y}ang-{M}ills equations in
  {$\bold R\sp {3+1}$}}, Ann. of Math. (2) \textbf{142} (1995), no.~1, 39--119.
  \MR{MR1338675 (96i:58167)}

\bibitem{KlainermanMachedon95}
\bysame, \emph{Smoothing estimates for null forms and applications}, Duke Math.
  J. \textbf{81} (1995), no.~1, 99--133 (1996), A celebration of John F. Nash,
  Jr. \MR{MR1381973 (97h:35022)}

\bibitem{Klainerman83}
Klainerman, S., \emph{Long time behaviour of solutions to nonlinear wave
  equations}, Proceedings of the {I}nternational {C}ongress of
  {M}athematicians, {V}ol.\ 1, 2 ({W}arsaw, 1983) (Warsaw), PWN, 1984,
  pp.~1209--1215. \MR{MR804771}

\bibitem{KM3}
Klainerman, S. and Machedon, M., \emph{Estimates for null forms and the spaces
  {$H\sb {s,\delta}$}}, Internat. Math. Res. Notices (1996), no.~17, 853--865.
  \MR{MR1420552 (98j:46028)}

\bibitem{KS2}
Klainerman, S. and Selberg, S., \emph{Remark on the optimal regularity for
  equations of wave maps type}, Comm. Partial Differential Equations
  \textbf{22} (1997), no.~5-6, 901--918. \MR{1452172}

\bibitem{KS}
\bysame, \emph{Bilinear estimates and applications to nonlinear wave
  equations}, Commun. Contemp. Math. \textbf{4} (2002), no.~2, 223--295.
  \MR{MR1901147 (2003d:35182)}

\bibitem{KrSchTa}
Krieger, J., Schlag, W., and Tataru, D., \emph{Renormalization and blow up for
  the critical {Y}ang-{M}ills problem}, Adv. Math. \textbf{221} (2009), no.~5,
  1445--1521. \MR{2522426}

\bibitem{KT_YM}
{Krieger}, J. and {Tataru}, D., \emph{{Global well-posedness for the Yang-Mills
  equation in $4+1$ dimensions. Small energy}}, ArXiv e-prints (2015).

\bibitem{Krieger_H2}
Krieger, J., \emph{Global regularity of wave maps from {$\bold R^{2+1}$} to
  {$H^2$}. {S}mall energy}, Comm. Math. Phys. \textbf{250} (2004), no.~3,
  507--580. \MR{2094472}

\bibitem{KL_MKG}
Krieger, J. and L\"uhrmann, J., \emph{Concentration compactness for the
  critical {M}axwell-{K}lein-{G}ordon equation}, Ann. PDE \textbf{1} (2015),
  no.~1, Art. 5, 208. \MR{3479062}

\bibitem{KriegerSchlag}
Krieger, J. and Schlag, W., \emph{Concentration compactness for critical wave
  maps}, EMS Monographs in Mathematics, European Mathematical Society (EMS),
  Z\"urich, 2012. \MR{2895939}

\bibitem{KST}
Krieger, J., Sterbenz, J., and Tataru, D., \emph{Global well-posedness for the
  {M}axwell-{K}lein-{G}ordon equation in {$4+1$} dimensions: small energy},
  Duke Math. J. \textbf{164} (2015), no.~6, 973--1040. \MR{3336839}

\bibitem{LRS}
Lebowitz, J., Rose, H., and Speer, E., \emph{Statistical mechanics of the
  nonlinear {S}chr\"odinger equation}, J. Statist. Phys. \textbf{50} (1988),
  no.~3-4, 657--687.

\bibitem{Lindblad96}
Lindblad, H., \emph{Counterexamples to local existence for semi-linear wave
  equations}, Amer. J. Math. \textbf{118} (1996), no.~1, 1--16. \MR{MR1375301
  (97b:35124)}

\bibitem{LM1}
L{\"u}hrmann, J. and Mendelson, D., \emph{Random data {C}auchy theory for
  nonlinear wave equations of power-type on {$\mathbb{R}^3$}}, Comm. Partial
  Differential Equations \textbf{39} (2014), no.~12, 2262--2283.

\bibitem{LM2}
L{\"u}hrmann, J. and Mendelson, D., \emph{On the almost sure global
  well-posedness of energy sub-critical nonlinear wave equations on {$\Bbb
  R^3$}}, New York J. Math. \textbf{22} (2016), 209--227. \MR{3484682}

\bibitem{MoncriefMKG}
Moncrief, V., \emph{Global existence of {M}axwell-{K}lein-{G}ordon fields in
  {$(2+1)$}-dimensional spacetime}, J. Math. Phys. \textbf{21} (1980), no.~8,
  2291--2296. \MR{579231 (82c:81089)}

\bibitem{NSU}
Nahmod, A., Stefanov, A., and Uhlenbeck, K., \emph{On the well-posedness of the
  wave map problem in high dimensions}, Comm. Anal. Geom. \textbf{11} (2003),
  no.~1, 49--83. \MR{2016196}

\bibitem{OT_MKG}
Oh, S.-J. and Tataru, D., \emph{Global well-posedness and scattering of the
  {$(4+1)$}-dimensional {M}axwell-{K}lein-{G}ordon equation}, Invent. Math.
  \textbf{205} (2016), no.~3, 781--877. \MR{3539926}

\bibitem{OP}
Oh, T. and Pocovnicu, O., \emph{Probabilistic global well-posedness of the
  energy-critical defocusing quintic nonlinear wave equation on {$\Bbb{R}^3$}},
  J. Math. Pures Appl. (9) \textbf{105} (2016), no.~3, 342--366. \MR{3465807}

\bibitem{Pecher_MKGtemp}
Pecher, H., \emph{Low regularity solutions for the {$(2+1)$}-dimensional
  {M}axwell-{K}lein-{G}ordon equations in temporal gauge}, Commun. Pure Appl.
  Anal. \textbf{15} (2016), no.~6, 2203--2219. \MR{3565939}

\bibitem{Pocovnicu}
Pocovnicu, O., \emph{{Almost sure global well-posedness for the energy-critical
  defocusing nonlinear wave equation on $\mathbb{R}^d$, $d=4$ and $5$}}, to
  appear in J. Eur. Math. Soc.

\bibitem{PonceSideris}
Ponce, G. and Sideris, T.~C., \emph{Local regularity of nonlinear wave
  equations in three space dimensions}, Comm. Partial Differential Equations
  \textbf{18} (1993), no.~1-2, 169--177. \MR{1211729}

\bibitem{RD69}
Rennie, B.~C. and Dobson, A.~J., \emph{On {S}tirling numbers of the second
  kind}, J. Combinatorial Theory \textbf{7} (1969), 116--121. \MR{0241310}

\bibitem{RodnianskiSterbenz}
Rodnianski, I. and Sterbenz, J., \emph{On the formation of singularities in the
  critical {${\rm O}(3)$} {$\sigma$}-model}, Ann. of Math. (2) \textbf{172}
  (2010), no.~1, 187--242. \MR{2680419}

\bibitem{Schlag_notes}
Schlag, W., \emph{{Harmonic Analysis}}.

\bibitem{SS_WM}
Shatah, J. and Struwe, M., \emph{The {C}auchy problem for wave maps}, Int.
  Math. Res. Not. (2002), no.~11, 555--571. \MR{1890048}

\bibitem{Suzzoni1}
\SortNoop{Suzzoni}de~Suzzoni, A., \emph{{Large data low regularity scattering
  results for the wave equation on the Euclidean space}}, {Commun. Partial
  Differ. Equations} \textbf{38} (2013), no.~1-3, 1--49.

\bibitem{Suzzoni2}
\bysame, \emph{Consequences of the choice of a particular basis of {$L^2(S^3)$}
  for the cubic wave equation on the sphere and the {E}uclidean space}, Commun.
  Pure Appl. Anal. \textbf{13} (2014), no.~3, 991--1015.

\bibitem{SterbenzTataru_WM2}
Sterbenz, J. and Tataru, D., \emph{Regularity of wave-maps in dimension
  {$2+1$}}, Comm. Math. Phys. \textbf{298} (2010), no.~1, 231--264.
  \MR{2657818}

\bibitem{Struwe}
Struwe, M., \emph{Equivariant wave maps in two space dimensions}, Comm. Pure
  Appl. Math. \textbf{56} (2003), no.~7, 815--823, Dedicated to the memory of
  J\"urgen K. Moser. \MR{1990477}

\bibitem{Tanguay}
Tanguay, A.~J., \emph{New bilinear estimates for quadratic-derivative nonlinear
  wave equations in 2+1 dimensions}, ProQuest LLC, Ann Arbor, MI, 2012, Thesis
  (Ph.D.)--University of Massachusetts Amherst. \MR{3122047}

\bibitem{Tao_WMII}
Tao, T., \emph{Global regularity of wave maps. {II}. {S}mall energy in two
  dimensions}, Comm. Math. Phys. \textbf{224} (2001), no.~2, 443--544.
  \MR{1869874}

\bibitem{Tesfahun}
Tesfahun, A., \emph{Almost critical local well-posedness for the space-time
  monopole equation in {L}orenz gauge}, Commun. Contemp. Math. \textbf{17}
  (2015), no.~3, 1450043, 14. \MR{3325046}

\bibitem{Tz10}
Tzvetkov, N., \emph{Construction of a {G}ibbs measure associated to the
  periodic {B}enjamin-{O}no equation}, Probab. Theory Related Fields
  \textbf{146} (2010), no.~3-4, 481--514.

\bibitem{Uhlenbeck_Lp}
Uhlenbeck, K.~K., \emph{Connections with {$L^{p}$}\ bounds on curvature}, Comm.
  Math. Phys. \textbf{83} (1982), no.~1, 31--42. \MR{648356}

\bibitem{Uhlenbeck_YM}
\bysame, \emph{Removable singularities in {Y}ang-{M}ills fields}, Comm. Math.
  Phys. \textbf{83} (1982), no.~1, 11--29. \MR{648355}

\bibitem{Ward88}
Ward, R.~S., \emph{Soliton solutions in an integrable chiral model in {$2+1$}
  dimensions}, J. Math. Phys. \textbf{29} (1988), no.~2, 386--389. \MR{MR927022
  (89h:81126)}

\bibitem{Ward89}
\bysame, \emph{Twistors in {$2+1$} dimensions}, J. Math. Phys. \textbf{30}
  (1989), no.~10, 2246--2251. \MR{MR1016291 (90k:32089)}

\bibitem{Wu_WWM}
Wu, D., \emph{The {C}auchy problem of the {W}ard equation}, J. Funct. Anal.
  \textbf{256} (2009), no.~1, 215--257. \MR{2475422}

\bibitem{ZF}
Zhang, T. and Fang, D., \emph{Random data {C}auchy theory for the generalized
  incompressible {N}avier-{S}tokes equations}, J. Math. Fluid Mech. \textbf{14}
  (2012), no.~2, 311--324.

\bibitem{Zh2}
Zhidkov, P.~E., \emph{An invariant measure for the nonlinear {S}chr\"odinger
  equation}, Dokl. Akad. Nauk SSSR \textbf{317} (1991), no.~3, 543--546.

\bibitem{Zh1}
Zhidkov, P.~E., \emph{On an infinite sequence of invariant measures for the
  cubic nonlinear {S}chr\"odinger equation}, Int. J. Math. Math. Sci.
  \textbf{28} (2001), no.~7, 375--394.

\bibitem{Zhou}
Zhou, Y., \emph{Local existence with minimal regularity for nonlinear wave
  equations}, Amer. J. Math. \textbf{119} (1997), no.~3, 671--703.
  \MR{MR1448218 (98e:35119)}

\end{thebibliography}

\end{document}